\documentclass[twoside]{article}

\def\shorttitle{Hyperbolic Systems Modeling Phase Transformations I}
\def\shortauthor{Y. Liu \ and \ M. Yamamoto}

\usepackage{amsfonts}
\usepackage{amsmath}
\usepackage{amssymb}
\usepackage{amscd}
\usepackage{amsbsy}
\usepackage{amsthm}
\usepackage{bm}
\usepackage{empheq}
\usepackage{graphicx}
\usepackage{psfrag}
\usepackage{color}
\usepackage{cite}
\usepackage{multicol}
\newfont{\myfnt}{cmssi10 scaled 1440}
\numberwithin{equation}{section}
\catcode`@=11
\def\ps@nk{\def\@oddhead{\vbox{\hbox to \hsize{\pic \footnotesize \it \shorttitle
\hfill \rm \thepage} \vspace{1mm} \vspace*{-2mm}}}
\def\@evenhead{\vbox{\hbox to \hsize{\pic \footnotesize \rm \thepage \hfill \it \shortauthor}
\vspace{1mm} \vspace*{-2mm}}}
\def\@oddfoot{} \def\@evenfoot{}}
\def\ps@first{\def\@oddhead{\vbox{\hbox to \hsize{\pic \footnotesize
} \break}}
\def\@oddfoot{} \def\@evenfoot{}}
\newtheoremstyle{thmstyle}
  {6pt}
  {6pt}
  {\it}
  {}
  {\bf}
  {}
  {.5em}
  {}
\newtheoremstyle{remstyle}
  {6pt}
  {6pt}
  {\rm}
  {}
  {\bf}
  {}
  {.5em}
  {}
\def\Section#1{\Sec{\large #1} \setcounter{equation}{0} \vskip -6mm \indent}
\def\Sec{\@Startsection{section}{1}{\z@}
                                   {-3.5ex \@plus -1ex \@minus -.2ex}%
                                   {2.3ex \@plus.2ex}%
                                   {\normalfont\large\bfseries\boldmath}}
\def\@Startsection#1#2#3#4#5#6{%
  \if@noskipsec \leavevmode \fi
  \par
  \@tempskipa #4\relax
  \@afterindenttrue
  \ifdim \@tempskipa <\z@
    \@tempskipa -\@tempskipa \@afterindentfalse
  \fi
  \if@nobreak
    \everypar{}%
  \else
    \addpenalty\@secpenalty\addvspace\@tempskipa
  \fi
  \@ifstar
    {\@ssect{#3}{#4}{#5}{#6}}%
    {\@dblarg{\@Sect{#1}{#2}{#3}{#4}{#5}{#6}}}}
\def\@Sect#1#2#3#4#5#6[#7]#8{%
  \ifnum #2>\c@secnumdepth
    \let\@svsec\@empty
  \else
    \refstepcounter{#1}%
    \protected@edef\@svsec{\@seccntformat{#1}\relax}%
  \fi
  \@tempskipa #5\relax
  \ifdim \@tempskipa>\z@
    \begingroup
      #6{%
          \@hangfrom{\hskip #3\relax\@svsec \hskip -2.5mm}%
          \interlinepenalty \@M #8\@@par}
    \endgroup
    \csname #1mark\endcsname{#7}%
    \addcontentsline{toc}{#1}{%
      \ifnum #2>\c@secnumdepth \else
        \protect\numberline{\csname the#1\endcsname}%
      \fi
      #7}%
  \else
    \def\@svsechd{%
      #6{\hskip #3\relax
      \@svsec #8}%
      \csname #1mark\endcsname{#7}%
      \addcontentsline{toc}{#1}{%
        \ifnum #2>\c@secnumdepth \else
          \protect\numberline{\csname the#1\endcsname}%
        \fi
        #7}}%
  \fi
  \@xsect{#5}}
\renewenvironment{abstract}{%
        \small
        \quotation
         \noindent {\bfseries \abstractname } }%
      {\if@twocolumn\else\endquotation\fi}

\def\Subsec{\@StartSubsection{subsection}{2}{\z@}%
                                     {-3.25ex\@plus -1ex \@minus -.2ex}%
                                     {1.5ex \@plus .2ex}%
                                     {\normalfont\normalsize\bfseries\boldmath}}
\def\@StartSubsection#1#2#3#4#5#6{%
  \if@noskipsec \leavevmode \fi
  \par
  \@tempskipa #4\relax
  \@afterindenttrue
  \ifdim \@tempskipa <\z@
    \@tempskipa -\@tempskipa \@afterindentfalse
  \fi
  \if@nobreak
    \everypar{}%
  \else
    \addpenalty\@secpenalty\addvspace\@tempskipa
  \fi
  \@ifstar
    {\@ssect{#3}{#4}{#5}{#6}}%
    {\@dblarg{\@SubSect{#1}{#2}{#3}{#4}{#5}{#6}}}}
\def\@SubSect#1#2#3#4#5#6[#7]#8{%
  \ifnum #2>\c@secnumdepth
    \let\@svsec\@empty
  \else
    \refstepcounter{#1}%
    \protected@edef\@svsec{\@seccntformat{#1}\relax}%
  \fi
  \@tempskipa #5\relax
  \ifdim \@tempskipa>\z@
    \begingroup
      #6{%
          \@hangfrom{\hskip #3\relax\@svsec\hskip -1.5mm}%
          \interlinepenalty \@M #8\@@par}
    \endgroup
    \csname #1mark\endcsname{#7}%
    \addcontentsline{toc}{#1}{%
      \ifnum #2>\c@secnumdepth \else
        \protect\numberline{\csname the#1\endcsname}%
      \fi
      #7}%
  \else
    \def\@svsechd{%
      #6{\hskip #3\relax
      \@svsec #8}%
      \csname #1mark\endcsname{#7}%
      \addcontentsline{toc}{#1}{%
        \ifnum #2>\c@secnumdepth \else
          \protect\numberline{\csname the#1\endcsname}%
        \fi
        #7}}%
  \fi
  \@xsect{#5}}
\def\list#1#2{\ifnum \@listdepth >5\relax \@toodeep \else \global
\advance \@listdepth\@ne \fi \rightmargin \z@ \listparindent\z@
\itemindent\z@ \csname @list\romannumeral\the\@listdepth\endcsname
\def\@itemlabel{#1}\let\makelabel\@mklab \@nmbrlistfalse #2\relax
\@trivlist \parskip 0pt \parindent\listparindent \advance \linewidth
-\rightmargin \advance\linewidth -\leftmargin \advance\@totalleftmargin
\leftmargin \parshape \@ne \@totalleftmargin \linewidth \ignorespaces}
\renewcommand{\@makecaption}[2]{\begin{center}#1. #2\end{center}}
\catcode`@=12 
\theoremstyle{thmstyle}
\newtheorem{thm}{\indent Theorem}[section]
\newtheorem{lem}[thm]{\indent Lemma}
\newtheorem{prop}[thm]{\indent Proposition}

\theoremstyle{remstyle}
\newtheorem{rem}{\indent \bf Remark}[section]

\newsavebox{\mygraphic}
\def\pic{\begin{picture}(0,0) \put(-210,-1250){\usebox{\mygraphic}} \end{picture}}
\newfont{\HUGEbf}{cmbx10 scaled 3500}
\definecolor{gray}{rgb}{0.9,0.9,0.9}
\def\thebibliography#1{\section*{\bf \large References}
\list{[\arabic{enumi}]} {\settowidth \labelwidth{[#1]} \leftmargin
\labelwidth \advance \leftmargin \labelsep \usecounter{enumi}}
\def\newblock{\hskip .11em plus .33em minus .07em} \footnotesize \sloppy \clubpenalty
4000 \widowpenalty 4000 \sfcode`\.=1000 \relax}

\def\e{\mathrm e}

\def\de{\delta}

\theoremstyle{definition}

\numberwithin{equation}{section}

\title{\Large \bf \boldmath\ \\
On the Multiple Hyperbolic Systems Modeling Phase Transformation Kinetics$^*$}

\author{\large Yikan LIU$^\dag$\qquad Masahiro YAMAMOTO$^\dag$}

\date{}

\begin{document}

\maketitle

\thispagestyle{first}
\renewcommand{\thefootnote}{\fnsymbol{footnote}}

\footnotetext{\hspace*{-5mm} \begin{tabular}{@{}r@{}p{13cm}@{}} &
Manuscript last updated: \today.\\
$^\dag$ & Graduate School of Mathematical Sciences, the University of Tokyo, Komaba 3-8-1, Meguro, Tokyo 153-8914, Japan. E-mail: ykliu@ms.u-tokyo.ac.jp, myama@ms.u-tokyo.ac.jp\\
$^*$ & Partly supported by the Japan-France joint research project (Japan Society for Promotion of Science and Centre National de la Recherche Scientifique).
\end{tabular}}

\renewcommand{\thefootnote}{\arabic{footnote}}

\begin{abstract}
We discuss Cahn's time cone method modeling phase transformation kinetics. The model equation by the time cone method is an integral equation in the space-time region. First we reduce it to a system of hyperbolic equations, and in the case of odd spatial dimensions, the reduced system is a multiple hyperbolic equation. Next we propose a numerical method for such a hyperbolic system. By means of alternating direction implicit methods, numerical simulations for practical forward problems are implemented with satisfactory accuracy and efficiency.  In particular, in the three dimensional case, our numerical method on basis of reduced multiple hyperbolic equation, is fast.

\vskip 4.5mm

\noindent\begin{tabular}{@{}l@{ }p{9cm}} {\bf Keywords } & phase transformation, Cahn's time cone method,
multiple hyperbolic equation, fast numerical method
\end{tabular}

\noindent{\bf AMS subject classifications}\ \ 74N05, 35L30, 65M06, 65M99
\end{abstract}

\baselineskip 14.5pt

\setlength{\parindent}{1.5em}

\setcounter{section}{0}

\Section{Introduction}\label{sec_intro}

Phase transformations such as the crystallization of liquids and materials are important kinetics arising in both spontaneous phenomena and artificial processes. In such transformations, nucleation and structure growth consist of the most determinant kinetics which greatly characterize the final mechanical properties. In retrospect, the earliest stochastic modeling of the phase transformation can trace back to Johnson-Mehl-Avrami-Kolmogorov theory (usually abbreviated as JMAK theory, see Kolmogorov~\cite{K-1937}, Johnson \& Mehl~\cite{JM-1939} and Avrami~\cite{A-1939,A-1940,A-1941}). These pioneering works were concerned with an infinite specimen without transformation initially, in which the random events of generation were expected to follow the Poisson distribution. Hence the fraction of phase transformations reads
\begin{equation}\label{eq_def-Poisson}
P=1-\e^{-u},
\end{equation}
where $u$ denotes the expectation of the generation events. More importantly, newborn nuclei were assumed to appear randomly in the remaining untransformed space with a constant expected nucleation rate, and each nucleus was supposed to grow radially at a constant speed until impingement.

Efforts on extending the original JMAK theory have then been devoted extensively in the last several decades (see, e.g., \cite{C-1956,J-1974,JC-1992}), and one of the most remarkable works should be attributed to Cahn~\cite{C-1996}, which inherits the Poisson distribution assumption on generation events but greatly polishes the model of nuclei growth. More precisely, instead of constants the nucleation rate is allowed to be time- and space-dependent while the growth speed can be time-dependent, written as $\alpha(\bm x,t)$ and $\rho(t)$ respectively. With these settings, in general spatial dimensions the expectation of generation events is modeled as
\begin{equation}\label{eq_def-u}
u(\bm x,t)=\int_{\Omega_\rho(\bm x,t)}\alpha(\bm y,s)\,\mathrm d\bm y\mathrm ds\quad(\bm x\in\mathbb R^d,\ t\ge0),
\end{equation}
where $\Omega_\rho(\bm x,t)$ denotes the so-called ``time cone'' defined as
\begin{equation}\label{eq_def-timecone}
\Omega_\rho(\bm x,t):=\{(\bm y,s)\mid0<s<t,\ |\bm y-\bm x|<r(t,s)\},\quad r(t,s):=\int_s^t\rho(\tau)\,\mathrm d\tau.
\end{equation}
Obviously, $r(t,s)$ stands for the radius of a transformed domain at time $t$ generated by a nucleus which was born at time $s$ without impingement. Therefore, a time cone $\Omega_\rho(\bm x,t)$ can be physically interpreted as the ensemble of all pairs $(\bm y,s)$ which would have caused transformation at $(\bm x,t)$. Especially, when $\alpha,\rho$ are positive constants and $d=3$, the phase transformation fraction can be easily calculated from \eqref{eq_def-u} and \eqref{eq_def-Poisson}, yielding the well-known JMAK equation
\[P(\bm x,t)=1-\exp\left(-\pi\alpha\rho^3\,t^4/3\right).\]
For an intuitive understanding of time cones, see Figure~\ref{fig_timecone}. As subsequent researches after Cahn's time cone method, we refer to \cite{BAC-2005,RV-2009}.
\begin{figure}[htbp]\centering
\includegraphics[width=\textwidth]{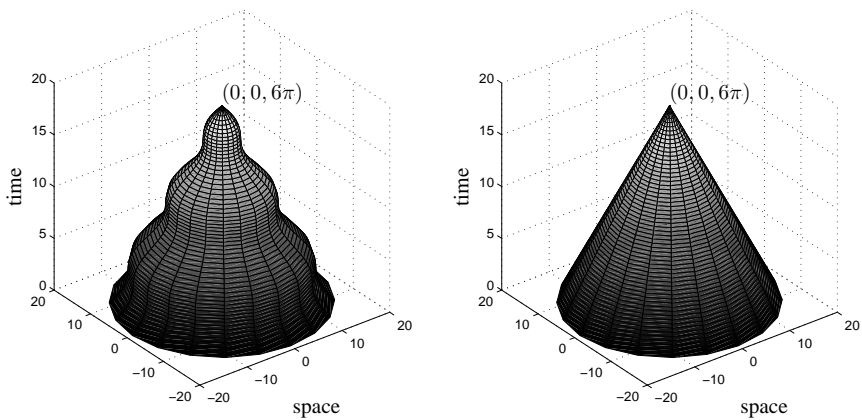}\\
\caption{\footnotesize{Examples of the two-dimensional time cones $\Omega_\rho(\bm x,t)$ with $\bm x=(0,0)$ and $t=6\pi$, generated by the growth speed $\rho(t)=1+0.9\cos t$ (left figure) and $\rho(t)\equiv1$ (right figure).}}\label{fig_timecone}
\end{figure}

Although models on phase transformation kinetics have been well established and have been widely utilized in industry, mathematical considerations on related forward and inverse problems are limited. To the best of our knowledge, only a similar but parallel concept named ``causal cone'' approach was proposed to study the morphology of crystalline polymeric systems, upon which several forward problems (see \cite{BC-2001,BCE-2002,BCS-2002,CS-2000,E-1996,EC-2005,MB-2001}) and inverse problems (see \cite{B-2001,BCE-1999,CEK-2008}) were investigated. For a comprehensive collection of mathematical topics on the polymer processing, we refer to Capasso~\cite{C-2003}. Recently, Liu, Xu and Yamamoto~\cite{LXY-2012} argued the one-dimensional identification of the growth speed on basis of Cahn's model.

As will be explained later, the time cone model \eqref{eq_def-u}--\eqref{eq_def-timecone} in its original expression is difficult to handle because it involves multiple integrations. Therefore, the purpose of this paper is to develop an alternative formulation describing Cahn's model which provides convenient methods for the discussion of both forward and inverse problems. Here by forward problems we mainly refer to finding $u$ by \eqref{eq_def-u} with given $\alpha$ and $\rho$ as well as suitable initial and boundary values, while inverse problems stand for the determination of $\alpha$ or $\rho$ by partial observations of $u$. The derived equivalent representations turn out to be a class of multiple hyperbolic systems, in the most concise forms only in odd spatial dimensions. Consequently, such treatment allows direct applications of abundant existing results concerning hyperbolic equations. We shall demonstrate the dramatically efficient forward solver in this paper and deal with several inverse problems in an upcoming one.

The rest of this paper is organized as follows. In Section~\ref{sec_main}, we briefly mention the motivation to find hyperbolic alternatives of Cahn's model \eqref{eq_def-u}--\eqref{eq_def-timecone} and state the main result, which proof is given in Section~\ref{sec_proof}. Section~\ref{sec_numer} shows numerical simulations of the forward problem in practical dimensions, and Section~\ref{sec_cclu} gives concluding remarks and prospections of future works. Finally, proofs of technical lemmata are postponed to Appendix~\ref{sec_tech}.

\Section{Motivation and Main Result}\label{sec_main}

From now on we concentrate on Cahn's time cone model \eqref{eq_def-u}--\eqref{eq_def-timecone}, which takes the form of an integral equation. More precisely, the nucleation rate $\alpha(\bm x,t)$ acts as the integrand function, and the growth speed $\rho(t)$ is embedded in the domain of integration. Therefore, although the solution $u$ is explicitly expressed in \eqref{eq_def-u}, in view of numerical treatments of forward problems it involves a $(d+1)$-dimensional numerical integration to approximate $u$ only for a single pair $(\bm x,t)$, not to mention the tremendous computational complexity in practice. On the other hand, note that the profile of a time cone $\Omega_\rho(\bm x,t)$ becomes irregular when the growth speed is no longer a constant (compare the left and right panels of Figure~\ref{fig_timecone}). Thus it is also  inconvenient to investigate corresponding inverse problems based on such an integral equation with a complicated domain of integration. These difficulties indicate the necessity to replace the original formulation by an equivalent time-evolutionary governing system, where $\alpha$ and $\rho$ are directly attainable.

In fact, such consideration is motivated by a first observation when $d=1$ and $\rho$ is a constant, in which case equation \eqref{eq_def-u} takes the exact form of d'Alembert's formula
\[u(x,t)=\int_0^t\!\!\int_{x-\rho(t-s)}^{x+\rho(t-s)}\alpha(y,s)\,\mathrm dy\mathrm ds.\]
In other words, providing certain regularity $\alpha\in C^{0,1}(\mathbb R\times\mathbb R_+)$, the function $u(x,t)$ should satisfy an inhomogeneous wave equation with homogeneous initial condition
\[\begin{cases}
(\partial_t^2-\rho^2\,\partial_x^2)u(x,t)=2\rho\,\alpha(x,t) & (x\in\mathbb R,\ t>0),\\
u(x,0)=\partial_tu(x,0)=0 & (x\in\mathbb R).
\end{cases}\]
Furthermore, obviously the growth speed and the nucleation rate play the roles of the propagation speed of wave and the source term (up to a multiplier) respectively. As a result, there is sufficient evidence to expect hyperbolic-type governing equations with respect to $u$ with time-dependent $\rho$ in higher spatial dimensions.

Now we state the main conclusion of the derived systems.

\begin{thm}[Multiple hyperbolic systems]\label{thm_gov-u}
Let the spatial dimensions $d=2m+1$ $(m=0,1,\ldots)$ and $u(\bm x,t)$ satisfy \eqref{eq_def-u}--\eqref{eq_def-timecone}. Assume that $u(\bm x,t),\ \alpha(\bm x,t)\ge0$ and $\rho(t)>0\ (\bm x\in\mathbb R^d,\ t\ge0)$ are sufficiently smooth functions, and introduce the hyperbolic operator
\[\mathcal P_\rho w(\bm x,t):=\frac1{\rho(t)}\partial_t\left(\frac{\partial_tw(\bm x,t)}{\rho(t)}\right)-\triangle w(\bm x,t).\]
Then $u(\bm x,t)$ satisfies the following multiple hyperbolic system
\begin{equation}\label{eq_gov-u-odd}
\begin{cases}
\mathcal P_\rho^{m+1}u(\bm x,t)=(2m)!!\,2^{m+1}\pi^m\,\alpha(\bm x,t)/\rho(t) & (\bm x\in\mathbb R^d,\ t>0),\\
\partial_t^ju(\bm x,0)=0,\quad j=0,1,\ldots,2m+1 & (\bm x\in\mathbb R^d).
\end{cases}
\end{equation}
\end{thm}

\begin{rem}
In the theorem, for simplicity, we assume that $\alpha$, $\rho$ and $u$ are sufficiently smooth. On the other hand, if $\rho\in C^d[0,T]$ and $\alpha\in L^2(\mathbb R^d\times(0,T))$ for some $T>0$, then by a priori estimates (e.g., Lions and Magenes~\cite{LM-1972}) and the multiple hyperbolic system \eqref{eq_gov-u-odd}, we can establish the corresponding regularity of $u$ in Sobolev spaces, but here we do not discuss the details.
\end{rem}

It is readily seen that $\mathcal P_\rho$ is a hyperbolic operator with a damping term and the propagation speed of wave is indeed $\rho(t)$. Actually, in the next section we can find a change of variable in time by which $\mathcal P_\rho$ corresponds to the d'Alembertian with respect to the new time axis. For any odd $d$, the above theorem indicates that the integral in the $d$-dimensional time cone model \eqref{eq_def-u} can be completely eliminated by acting $(d+1)/2$ times of the operator $\mathcal P_\rho$ to both sides. For instance, we obtain a single hyperbolic system for $d=1$ and an interesting double hyperbolic system for $d=3$. Moreover, in these multiple hyperbolic systems, $\alpha/\rho$ appears explicitly as the source term (up to a multiplier), and the initial conditions are always homogeneous. Unfortunately, such concise expressions as \eqref{eq_gov-u-odd} are unavailable for any even $d$. In these cases, it can be inferred from Proposition~\ref{prop_gov-Um} that at best $\mathcal P_\rho^{d/2}u$ equals $d/2$ terms of integrals concerning $\alpha$ and $\rho$ which cannot be further canceled. As will be witnessed in Section~\ref{sec_numer}, this drawback remains certain inconvenience even in the numerical simulation of the two-dimensional forward problem. The apparent difference between odd and even dimensions can be explained by Huygens' principle (see Remark~\ref{rem_Duhamel}).

\Section{Proof of the Main Results}\label{sec_proof}

In order to deal with the physical model \eqref{eq_def-u}--\eqref{eq_def-timecone} in general spatial dimensions, we start from some overall settings. Throughout this section we adopt the smoothness and positivity assumptions on $\alpha$ and $\rho$ in Theorem~\ref{thm_gov-u}. Denote $\bm x=(x_1,\ldots,x_d)\in\mathbb R^d$ and let
\[B_d(\bm x,\ell):=\{y\in\mathbb R^d\mid|\bm y-\bm x|<\ell\},\quad S_d(\bm x,\ell):=\partial B_d(\bm x,\ell)\]
be the open ball and the corresponding sphere centered at $\bm x$ with radius $\ell>0$. Then equation \eqref{eq_def-u} becomes
\begin{equation}\label{eq_def-u'}
u(\bm x,t)=\int_0^t\!\!\int_{B_d(\bm x,r(t,s))}\alpha(\bm y,s)\,\mathrm d\bm y\mathrm ds\quad(\bm x\in\mathbb R^d,\ t\ge0).
\end{equation}
Recall that $\rho(t)$ is not a constant in general which generates the irregularity of the domain $\Omega_\rho(\bm x,t)$ of integration. However, this difficulty can be overcome by introducing the change of variable in time
\begin{equation}\label{eq_change-var}
\tau=R(t):=\int_0^t\rho(s)\,\mathrm ds=r(t,0)\quad(t\ge0),
\end{equation}
which is also adopted in Cannon \cite{C-1984} to treat parabolic equations. Thanks to the strict positivity of $\rho$, the function $R(t)$ is nonnegative and strictly increasing for $t\ge0$, allowing a well-defined inverse function $t=R^{-1}(\tau)$ for $\tau\ge0$. Moreover, it turns out from taking derivative in the identity $R(R^{-1}(\tau))=\tau$ that $(R^{-1}(\tau))'=1/\rho(R^{-1}(\tau))$. Therefore, performing the same change of variable in the integral on its right-hand side, we further simplify \eqref{eq_def-u'} as
\begin{equation}\label{eq_def-U0}
U_0(\bm x,\tau):= u(\bm x,R^{-1}(\tau))=\int_0^\tau\!\!\int_{B_d(\bm x,\tau-\zeta)}\frac{\alpha(\bm y,R^{-1}(\zeta))}{\rho(R^{-1}(\zeta))}\,\mathrm d\bm y\mathrm d\zeta\quad(\bm x\in\mathbb R^d,\ \tau\ge0).
\end{equation}
Consequently, it is convenient to consider $U_0(\bm x,\tau)$ instead of $u(\bm x,t)$ hereinafter since now the integration is taken in a regular cone $\Omega_1(\bm x,\tau)$ with vertex $(\bm x,\tau)$ and unit slope (see the right figure of Figure~\ref{fig_timecone}). In fact, for any smooth function $w$ in $\mathbb R^d\times[0,\infty)$, we discover by simple calculations that the same change of variable \eqref{eq_change-var} and the definition $W(\bm x,\tau):= w(\bm x,R^{-1}(\tau))$ give
\[\partial_\tau^2W(\bm x,\tau)=\left.\frac1{\rho(t)}\partial_t\left(\frac{\partial_tw(\bm x,t)}{\rho(t)}\right)\right|_{t=R^{-1}(\tau)},\quad\partial_\tau W(\bm x,0)=\frac{\partial_tw(\bm x,0)}{\rho(0)},\]
or equivalently, by taking $\tau=R(t)$ and recalling the operator $\mathcal P_\rho$ in Theorem~\ref{thm_gov-u},
\begin{equation}\label{eq_change-var-id}
\begin{cases}
\!\begin{alignedat}{2}
& \mathcal P_\rho w(\bm x,t)=\square W(\bm x,R(t)) & \quad & (\bm x\in\mathbb R^d,\ t>0),\\
& w(\bm x,0)=W(\bm x,0),\ \partial_tw(\bm x,0)=\rho(0)\,\partial_\tau W(\bm x,0) & \quad & (\bm x\in\mathbb R^d),
\end{alignedat}
\end{cases}
\end{equation}
where $\square:=\partial_\tau^2-\triangle$ denotes the d'Alembertian with $\tau$ as the time variable.

For later convenience, we denote by $\sigma_d$ the surface area of the $d$-dimensional unit ball, and write $F(\bm x,\tau):=\alpha(\bm x,R^{-1}(\tau))/\rho(R^{-1}(\tau))$ for simplicity. Then we introduce the following integral brackets for $k,j=0,1,\ldots$, $\bm x\in\mathbb R^d$ and $\tau>0$ that
\begin{align}
[k,S_d,\triangle^j](\bm x,\tau) & :=\int_0^\tau\!\!\int_{S_d(\bm x,\tau-\zeta)}\frac{\triangle^jF(\bm y,\zeta)}{(\tau-\zeta)^k}\,\mathrm d\bm\sigma\mathrm d\zeta\quad(k\le d-1),\label{eq_def-brkt-bdry}\\
[k,B_d,\triangle^j](\bm x,\tau) & :=\int_0^\tau\!\!\int_{B_d(\bm x,\tau-\zeta)}\frac{\triangle^jF(\bm y,\zeta)}{(\tau-\zeta)^k}\,\mathrm d\bm y\mathrm d\zeta\quad(k\le d).\label{eq_def-brkt-inter}
\end{align}
The restriction on $k$ guarantees the well-posedness of the above definitions, that is, there is no singularity near $\tau=0$. Furthermore, by the smoothness assumption and an averaging argument, we find
\[\lim_{\tau\downarrow0}[k,S_d,\triangle^j](\bm x,\tau)=0\ (k\le d-1),\quad\lim_{\tau\downarrow0}[k,B_d,\triangle^j](\bm x,\tau)=0\ (k\le d),\]
which allows the redefinition
\begin{equation}\label{eq_brkt-initial}
[k,S_d,\triangle^j](\bm x,0)=0\ (k\le d-1),\quad[k,B_d,\triangle^j](\bm x,0)=0\ (k\le d).
\end{equation}

Now we relate the two brackets by differential operations.

\begin{lem}\label{lem_brkt-diff}
Let the spatial dimensions $d\ge2,\ k=0,1,\ldots,d-1$ and $j=0,1,\ldots$. Let $[k,S_d,\triangle^j](\bm x,\tau)$ and $[k,B_d,\triangle^j](\bm x,\tau)\ (\bm x\in\mathbb R^d,\ \tau>0)$ be defined as in $\eqref{eq_def-brkt-bdry}$ and $\eqref{eq_def-brkt-inter}$ respectively. Then
\begin{align}
& \triangle[k,S_d,\triangle^j]=[k,S_d,\triangle^{j+1}],\quad\triangle[k,B_d,\triangle^j]=[k,B_d,\triangle^{j+1}],\label{eq_brkt-laplace}\\
& \partial_\tau[k,S_d,\triangle^j]=\begin{cases}
(d-k-1)\,[k+1,S_d,\triangle^j]+[k,B_d,\triangle^{j+1}] & (k<d-1),\\
\sigma_d\,\triangle^jF+[d-1,B_d,\triangle^{j+1}] & (k=d-1),
\end{cases}\label{eq_brkt-timediff-bdry}\\
& \partial_\tau[k,B_d,\triangle^j]=-k\,[k+1,B_d,\triangle^j]+[k,S_d,\triangle^j].\label{eq_brkt-timediff-inter}
\end{align}
\end{lem}

The proof involves only elementary calculations and it will be given in Appendix~\ref{sec_tech}. Now we are able to state the first conclusion.

\begin{lem}[Single hyperbolic systems]\label{lem_single-hyper}
Let $d=1,2,\ldots$. Then

$(1)$\ \ $U_0(\bm x,\tau)$ defined in \eqref{eq_def-U0} satisfies
\begin{equation}\label{eq_gov-U0}
\begin{cases}
\!\begin{alignedat}{2}
& \square U_0(\bm x,\tau)=\begin{cases}
2\,F(x,\tau) & (d=1),\\
(d-1)\,U_1(\bm x,\tau) & (d\ge2)
\end{cases} & \quad & (\bm x\in\mathbb R^d,\ t>0),\\
& U_0(\bm x,0)=\partial_\tau U_0(\bm x,0)=0 & \quad & (\bm x\in\mathbb R^d),
\end{alignedat}
\end{cases}
\end{equation}
where
\begin{equation}\label{eq_def-U1}
U_1(\bm x,\tau):=[1,S_d,\triangle^0](\bm x,\tau)\quad(d\ge2).
\end{equation}

$(2)$\ \ $u(\bm x,t)$ in \eqref{eq_def-u'} satisfies
\begin{equation}\label{eq_gov-u0}
\begin{cases}
\!\begin{alignedat}{2}
& \mathcal P_\rho u(\bm x,t)=\begin{cases}
2\,\alpha(x,t)/\rho(t) & (d=1),\\
(d-1)\,U_1(\bm x,R(t)) & (d\ge2)
\end{cases} & \quad & (\bm x\in\mathbb R^d,\ t>0),\\
& u(\bm x,0)=\partial_tu(\bm x,0)=0 & \quad & (\bm x\in\mathbb R^d).
\end{alignedat}
\end{cases}
\end{equation}
\end{lem}

\begin{proof}
(1)\ \ For $d=1$, we return to the original definition \eqref{eq_def-U0} and write
\[U_0(x,\tau)=\int_0^\tau\!\!\int_{x-(\tau-\zeta)}^{x+(\tau-\zeta)}F(y,\zeta)\,\mathrm d y\mathrm d\zeta,\]
following the fundamental differentiations
\begin{align*}
\partial_xU_0(x,\tau) & =\int_0^\tau\left(F(x+(\tau-\zeta),\zeta)-F(x-(\tau-\zeta),\zeta)\right)\mathrm d\zeta,\\
\partial_x^2U_0(x,\tau) & =\int_0^\tau\left(\partial_xF(x+(\tau-\zeta),\zeta)-\partial_xF(x-(\tau-\zeta),\zeta)\right)\mathrm d\zeta,\\
\partial_\tau U_0(x,\tau) & =\int_0^\tau\left(F(x+(\tau-\zeta),\zeta)+F(x-(\tau-\zeta),\zeta)\right)\mathrm d\zeta,\\
\partial_\tau^2U_0(x,\tau) & =2\,F(x,\tau)+\int_0^\tau\left(\partial_xF(x+(\tau-\zeta),\zeta)-\partial_xF(x-(\tau-\zeta),\zeta)\right)\mathrm d\zeta\\
& =\partial_x^2U_0(x,\tau)+2\,F(x,\tau).
\end{align*}
On the other hand, the homogeneous initial condition is easily checked for $d=1$.

Considering dimensions $d\ge2$, we recognize $U_0(\bm x,\tau)=[0,B_d,\triangle^0](\bm x,\tau)$ and apply Lemma~\ref{lem_brkt-diff} with $k=j=0$ to obtain
\begin{align*}
\partial_\tau U_0(\bm x,\tau) & =\partial_\tau[0,B_d,\triangle^0](\bm x,\tau)=[0,S_d,\triangle^0](\bm x,\tau),\\
\partial_\tau^2U_0(\bm x,\tau) & =\partial_t[0,S_d,\triangle^0](\bm x,\tau)=(d-1)\,[1,S_d,\triangle^0](\bm x,\tau)+[0,B_d,\triangle^1](\bm x,\tau)\\
& =\triangle U_0(\bm x,\tau)+(d-1)\,U_1(\bm x,\tau).
\end{align*}
Simultaneously, it follows from \eqref{eq_brkt-initial} that the initial condition is still homogeneous for $d\ge2$. This completes the verification of \eqref{eq_gov-U0}.

(2)\ \ The substitution of $w=u$ and $W=U_0$ in relation \eqref{eq_change-var-id} yields \eqref{eq_gov-u0} immediately from the above result.
\end{proof}

\begin{rem}\label{rem_Duhamel}
The above lemma demonstrates Theorem~\ref{thm_gov-u} for $d=1$ and suggests an inductive approach to higher dimensions. Although one may apply a d'Alembertian once more to $U_1(\bm x,\tau)$ for $d\ge3$ to obtain similar wave equations, another observation provides a straightforward reasoning. Write
\begin{align}
& U_1(\bm x,\tau)=\int_0^\tau V_1(\bm x,\tau;\zeta)\,\mathrm d\zeta\quad\mbox{with}\label{eq_Duhamel}\\
& V_1(\bm x,\tau;\zeta):=\frac1{\tau-\zeta}\int_{S_d(\bm x,\tau-\zeta)}F(\bm y,\zeta)\,\mathrm d\bm\sigma\quad(d\ge2).\label{eq_def-V1}
\end{align}
In view of Duhamel's principle (see, e.g., Evans~\cite{E-2010}), $U_1$ and $V_1$ satisfy the same type of equation with corresponding inhomogeneous right-hand term and initial condition.

(1)\ \ Especially, we claim for $d=3$ that $V_1(\bm x,\tau;\zeta)$ is of the form \eqref{eq_def-V1} if and only if
\begin{equation}\label{eq_Poisson}
\begin{cases}
\square V_1(\bm x,\tau;\zeta)=0 & (\bm x\in\mathbb R^d,\ \tau>\zeta),\\
V_1(\bm x,\tau;\zeta)|_{\tau=\zeta}=0,\ \partial_\tau V_1(\bm x,\tau;\zeta)|_{\tau=\zeta}=4\pi\,F(\bm x,\zeta) & (\bm x\in\mathbb R^d).
\end{cases}
\end{equation}
Actually, under the translation $\xi=\tau-\zeta$, \eqref{eq_Poisson} with $d=3$ is equivalent to
\[\begin{cases}
(\partial_\xi^2-\triangle)V_1(\bm x,\xi+\zeta;\zeta)=0 & (\bm x\in\mathbb R^3,\ \xi>0),\\
V_1(\bm x,\xi+\zeta;\zeta)|_{\xi=0}=0,\ \partial_\xi V_1(\bm x,\xi+\zeta;\zeta)|_{\xi=0}=4\pi\,F(\bm x,\zeta) & (\bm x\in\mathbb R^3).
\end{cases}\]
Noting that the above system is now independent of $\zeta$, we may apply Poisson's formula for the Cauchy problem of the three-dimensional wave equation to obtain
\[V_1(\bm x,\xi+\zeta;\zeta)=\frac1\xi\int_{S_3(\bm x,\xi)}F(\bm y,\zeta)\,\mathrm d\bm\sigma,\]
which is exactly \eqref{eq_def-V1} by replacing $\xi$ with $\tau-\zeta$. On the other hand, Duhamel's principle implies that under the relation \eqref{eq_Duhamel}, system \eqref{eq_Poisson} holds for $V_1(\bm x,\tau;\zeta)$ if and only if $U_1(\bm x,\tau)$ satisfies a wave equation for $d=3$. Consequently, together with Lemma~\ref{lem_single-hyper}(1), it turns out that $U_0(\bm x,t)$ satisfies a double wave equation and thus Theorem~\ref{thm_gov-u} for $d=3$ follows, stimulating the further discussion in higher spatial dimensions.

(2)\ \ However, it follows from \cite[\S2.4.1]{E-2010} that \eqref{eq_def-V1} cannot be the solution to \eqref{eq_Poisson} in even dimensions. Actually, for even $d$ the solution $V_1(\,\cdot\,,\,\cdot\,;\zeta)$ to \eqref{eq_Poisson} is affected by $F(\,\cdot\,,\zeta)$ inside the cone $\{(\bm y,\tau)\mid\tau>\zeta,\ |\bm y-\bm x|<\tau-\zeta\}$, while $V_1$ in \eqref{eq_def-V1} only on the lateral. This indeed coincides with Huygens' principle, namely, functions depending only on a sharp wavefront in even dimensions do not satisfy wave equations.
\end{rem}

\begin{prop}\label{prop_gov-Um}
Let $d\ge2m+1$ with $m=0,1,\ldots$ and $U_0(\bm x,\tau)$ be defined as in \eqref{eq_def-U0}. Then there holds
\begin{equation}\label{eq_gov-Um}
\begin{cases}
\!\begin{alignedat}{2}
& \square U_m=\begin{cases}
2^{m+1}\pi^m\,F & (d=2m+1),\\
(d-(2m+1))\,U_{m+1} & (d>2m+1)
\end{cases} & \quad & \mbox{in }\mathbb R^d\times\mathbb R_+,\\
& U_m(\,\cdot\,,0)=\partial_\tau U_m(\,\cdot\,,0)=0 & \quad & \mbox{in }\mathbb R^d,
\end{alignedat}
\end{cases}
\end{equation}
where we have for $m\ge1$ that
\begin{align}
& U_m=\sum_{k=1}^mc_m^k\,P_m^k(d)\,[2m-k,\partial^{(1-(-1)^k)/2}B_d,\triangle^{\lfloor k/2\rfloor}],\quad\mbox{in particular}\label{eq_exp-Um}\\
& P_m^m(d)=1,\quad P_m^k(d)=(d-2(m-\lfloor(k+1)/2\rfloor))\,P_{m-1}^k(d)\ (1\le k\le m-1),\label{eq_exp-Pmk}\\
& c_m^1=c_m^m=1,\quad c_m^k=\begin{cases}
c_{m-1}^{k-1} & (k\mbox{ even}),\\
c_{m-1}^{k-1}+c_{m-1}^k & (k\mbox{ odd})
\end{cases}\ (2\le k\le m-1).\label{eq_exp-cmk}
\end{align}
Here we understand $\partial^1B_d=S_d,\ \partial^0B_d=B_d,\ \lfloor\,\cdot\,\rfloor$ denotes the integer part of a positive number, and those terms without definitions automatically vanish.
\end{prop}

The verification of the above conclusion requires a technical lemma, and the proof is postponed to Appendix~\ref{sec_tech}.

\begin{lem}\label{lem_tech}
Let the integers $c_m^k\ (m=1,2,\ldots,\ 1\le k\le m)$ be defined as in $\eqref{eq_exp-Pmk}$ and $\eqref{eq_exp-cmk}$. Then

$(1)$\ \ For $m\ge2$ and $2\le k\le m,$ we have
\begin{equation}\label{eq_rel-Pmk}
P_m^{k-1}(d)=((d-m)+(-1)^k(m-2\lfloor k/2\rfloor))\,P_m^k(d).
\end{equation}

$(2)$\ \ For $m\ge3$ and $2\le k\le m-1,$ we have
\begin{equation}\label{eq_rel-cmk}
2(m-k)\,c_m^k=\begin{cases}
k\,c_m^{k+1} & (k\mbox{ even}),\\
(2m-k-1)\,c_m^{k+1} & (k\mbox{ odd}).
\end{cases}
\end{equation}
\end{lem}

\noindent{\it Proof of Proposition~$\ref{prop_gov-Um}$.} It is natural to adopt an inductive argument since the result for $m=0$ has been proved in Lemma~\ref{lem_single-hyper}(1). Thus it suffices to show for some $m\ge1$ that

(a)\ \ $U_m$ in \eqref{eq_exp-Um}--\eqref{eq_exp-cmk} satisfies the wave system \eqref{eq_gov-Um}, and

(b)\ \ for $d>2m+1$, $\square U_m/(d-(2m+1))$ preserves expression \eqref{eq_exp-Um}--\eqref{eq_exp-cmk} for $m+1$.

\noindent To this end, first we unify \eqref{eq_brkt-timediff-bdry}--\eqref{eq_brkt-timediff-inter} in Lemma~\ref{lem_brkt-diff} succinctly and substitute $k$ with $2m-k$ to derive
\begin{align*}
& \quad\,\,\partial_\tau[2m-k,\partial^{(1-(-1)^k)/2}B_d,\triangle^{\lfloor k/2\rfloor}]\\
& =\left((d-1)\frac{1-(-1)^k}2-2m+k\right)[2m-k+1,\partial^{(1-(-1)^k)/2}B_d,\triangle^{\lfloor k/2\rfloor}]\\
& \quad\,+[2m-k,\partial^{(1-(-1)^{k+1})/2}B_d,\triangle^{\lfloor(k+1)/2\rfloor}],
\end{align*}
yielding
\begin{align*}
& \quad\,\,\partial_\tau U_m=\sum_{k=1}^mc_m^k\,P_m^k(d)\,\partial_\tau[2m-k,\partial^{(1-(-1)^k)/2}B_d,\triangle^{\lfloor k/2\rfloor}]\\
& =\sum_{k=1}^mc_m^k\,P_m^k(d)\left\{\left((d-1)\frac{1-(-1)^k}2-2m+k\right)[2m-k+1,\partial^{(1-(-1)^k)/2}B_d,\triangle^{\lfloor k/2\rfloor}]\right.\\
& \qquad\qquad\qquad\quad\;\;\,+[2m-k,\partial^{(1-(-1)^{k+1})/2}B_d,\triangle^{\lfloor(k+1)/2\rfloor}]\bigg\}\\
& =(d-2m)\,P_m^1(d)\,[2m,S_d,\triangle^0]\\
& \quad\,+\sum_{k=2}^mc_m^k\,P_m^k(d)\left((d-1)\frac{1-(-1)^k}2-2m+k\right)[2m-k+1,\partial^{(1-(-1)^k)/2}B_d,\triangle^{\lfloor k/2\rfloor}]\\
& \quad\,+\sum_{k=2}^mc_m^{k-1}\,P_m^{k-1}(d)\,[2m-k+1,\partial^{(1-(-1)^k)/2}B_d,\triangle^{\lfloor k/2\rfloor}]\\
& \quad\,+[m,\partial^{(1-(-1)^{m+1})/2}B_d,\triangle^{\lfloor(m+1)/2\rfloor}]\\
& =P_{m+1}^1(d)\,[2m,S_d,\triangle^0]+\widehat U_m,
\end{align*}
where
\begin{align*}
\widehat U_m & :=\sum_{k=2}^mQ_m^k(d)\,[2m-k+1,\partial^{(1-(-1)^k)/2}B_d,\triangle^{\lfloor k/2\rfloor}]\\
& \quad\,+[m,\partial^{(1-(-1)^{m+1})/2}B_d,\triangle^{\lfloor(m+1)/2\rfloor}],
\end{align*}
in particular
\[Q_m^k(d):= c_m^k\,P_m^k(d)\left((d-1)\frac{1-(-1)^k}2-2m+k\right)+c_m^{k-1}\,P_m^{k-1}(d)\quad(k=2,\ldots,m).\]
Here we have applied \eqref{eq_exp-Pmk} with $m$ replaced by $m+1$ to get $(d-2m)\,P_m^1(d)=P_{m+1}^1(d)$. Meanwhile, using the fact that $d\ge2m+1$, we may apply \eqref{eq_brkt-initial} to argue that each integral bracket in $U_m$ and $\partial_\tau U_m$ vanish at $\tau=0$ and hence $\eqref{eq_gov-Um}_2$ holds.

Furthermore, we employ a similar argument for $\widehat U_m$ to obtain
\begin{align*}
& \quad\,\,\partial_\tau\widehat U_m\\
& =\sum_{k=2}^mQ_m^k(d)\left\{\left((d-1)\frac{1-(-1)^k}2-2m+k-1\right)[2m-k+2,\partial^{(1-(-1)^k)/2}B_d,\triangle^{\lfloor k/2\rfloor}]\right.\\
& \qquad\qquad\qquad\;\:+[2m-k+1,\partial^{(1-(-1)^{k+1})/2}B_d,\triangle^{\lfloor(k+1)/2\rfloor}]\bigg\}\\
& \quad\,+\left((d-1)\frac{1-(-1)^{m+1}}2-m\right)[m+1,\partial^{(1-(-1)^{m+1})/2}B_d,\triangle^{\lfloor(m+1)/2\rfloor}]\\
& \quad\,+[m,\partial^{(1-(-1)^m)/2}B_d,\triangle^{\lfloor m/2\rfloor+1}]\\
& =-(2m-1)\,P_{m+1}^2(d)\,[2m,B_d,\triangle^1]\\
& \quad\,+\sum_{k=1}^{m-2}\left\{Q_m^{k+2}(d)\left((d-1)\frac{1-(-1)^k}2-2m+k+1\right)+Q_m^{k+1}(d)\right\}\\
& \qquad\qquad\:\!\times[2m-k,\partial^{(1-(-1)^k)/2}B_d,\triangle^{\lfloor k/2\rfloor+1}]\\
& \quad\,+\left\{Q_m^m(d)+\left((d-1)\frac{1-(-1)^{m-1}}2-m\right)\right\}[m+1,\partial^{(1-(-1)^{m-1})/2}B_d,\triangle^{\lfloor(m-1)/2\rfloor+1}]\\
& \quad\,+[m,\partial^{(1-(-1)^m)/2}B_d,\triangle^{\lfloor m/2\rfloor+1}],
\end{align*}
where we have used \eqref{eq_exp-cmk}, \eqref{eq_exp-Pmk} and Lemma~\ref{lem_tech}(1) to find
\[Q_m^2(d)=c_m^2\,P_m^2(d)\,(-2m+2)+c_m^1\,P_m^1(d)=(d-2m)\,P_m^2(d)=P_{m+1}^2(d).\]
On the other hand, we differentiate $[2m,S_d,\triangle^0]$ with respect to $d$ to proceed
\begin{align*}
\partial_\tau^2U_m & =P_{m+1}^1(d)\,\partial_\tau[2m,S_d,\triangle^0]+\partial_\tau\widehat U_m\\
& =\begin{cases}
\!\begin{alignedat}{2}
& P_{m+1}^1(d)\left(\sigma_d\,F+[2m,B_d,\triangle^1]\right)+\partial_\tau\widehat U_m & & (d=2m+1),\\
& P_{m+1}^1(d)\left((d-2m-1)\,[2m+1,S_d,\triangle^0]+[2m,B_d,\triangle^1]\right)\\
& +\partial_\tau\widehat U_m & \quad & (d>2m+1),
\end{alignedat}
\end{cases}
\end{align*}
while
\[\triangle U_m=\sum_{k=1}^mc_m^k\,P_m^k(d)\,[2m-k,\partial^{(1-(-1)^k)/2}B_d,\triangle^{\lfloor k/2\rfloor+1}].\]
Representing $\eqref{eq_gov-Um}_1$ by the above expressions and comparing the both sides, we claim that it suffices to prove for $d\ge2m+1$ that
\begin{equation}\label{eq_claim-1}
\begin{aligned}
& \sigma_{2m+1}\,P_{m+1}^1(2m+1)=2^{m+1}\pi^m,\\
& P_{m+1}^1(d)-(2m-1)\,P_{m+1}^2(d)=(d-2m-1)\,c_{m+1}^2\,P_{m+1}^2(d)
\end{aligned}\quad(m\ge1),
\end{equation}
and
\begin{align}
& \quad\,\,Q_m^m(d)+\left((d-1)\frac{1-(-1)^{m-1}}2-m\right)-c_m^{m-1}\,P_m^{m-1}(d)\nonumber\\
& =(d-2m-1)\,c_{m+1}^{m+1}\,P_{m+1}^{m+1}(d)\quad(m\ge2),\label{eq_claim-2}\\
& \quad\,\,Q_m^{k+2}(d)\left((d-1)\frac{1-(-1)^k}2-2m+k+1\right)+Q_m^{k+1}(d)-c_m^k\,P_m^k(d)\nonumber\\
& =(d-2m-1)\,c_{m+1}^{k+2}\,P_{m+1}^{k+2}(d)\quad(m\ge3,\ k=1,\ldots,m-2).\label{eq_claim-3}
\end{align}
In fact, as long as \eqref{eq_claim-1}--\eqref{eq_claim-3} are valid, requirements (a) and (b) are satisfied simultaneously and the proof is complete.

First, the repeated applications of Lemma~\ref{lem_tech}(1) with $m$ replaced by $m+1$ yields
\[P_{m+1}^1(d)=(d-2)\,P_{m+1}^2(d)=(d-2)(d-2m)\,P_{m+1}^3(d)=\cdots=\prod_{j=1}^m(d-2j),\]
which, together with the fact $\sigma_{2m+1}=2^{m+1}\pi^m/(2m-1)!!$, leads to
\[\sigma_{2m+1}\,P_{m+1}^1(2m+1)=\frac{2^{m+1}\pi^m}{(2m-1)!!}\prod_{j=1}^m(2m+1-2j)=2^{m+1}\pi^m\]
and meanwhile
\begin{align*}
P_{m+1}^1(d)-(2m-1)\,P_{m+1}^2(d) & =(d-2)\,P_{m+1}^2(d)-(2m-1)\,P_{m+1}^2(d)\\
& =(d-2m-1)\,c_{m+1}^2\,P_{m+1}^2(d),
\end{align*}
that is, \eqref{eq_claim-1}. Next, it follows from the expansion of $Q_m^m(d)$ that
\begin{align*}
& \quad\,\,Q_m^m(d)+\left((d-1)\frac{1-(-1)^{m-1}}2-m\right)-c_m^{m-1}\,P_m^{m-1}(d)\\
& =\left((d-1)\frac{1-(-1)^m}2-m\right)+c_m^{m-1}\,P_m^{m-1}(d)+\left((d-1)\frac{1-(-1)^{m-1}}2-m\right)\\
& \quad\,-c_m^{m-1}\,P_m^{m-1}(d)\\
& =d-2m-1=(d-2m-1)\,c_{m+1}^{m+1}\,P_{m+1}^{m+1}(d)
\end{align*}
or \eqref{eq_claim-2}. Similarly, we expand $Q_m^{k+1}(d)$ and $Q_m^{k+2}(d)$ for $k=1,\ldots,m-2$ and utilize Lemma~\ref{lem_tech}(1) to calculate
\begin{align*}
& \quad\,\,Q_m^{k+2}(d)\left((d-1)\frac{1-(-1)^k}2-2m+k+1\right)+Q_m^{k+1}(d)-c_m^k\,P_m^k(d)\\
& =\left\{c_m^{k+2}\,P_m^{k+2}(d)\left((d-1)\frac{1-(-1)^k}2-2m+k+2\right)+c_m^{k+1}\,P_m^{k+1}(d)\right\}\\
& \quad\,\times\left((d-1)\frac{1-(-1)^k}2-2m+k+1\right)\\
& \quad\,+\left\{c_m^{k+1}\,P_m^{k+1}(d)\left((d-1)\frac{1-(-1)^{k+1}}2-2m+k+1\right)+c_m^k\,P_m^k(d)\right\}-c_m^k\,P_m^k(d)\\
& =\left\{c_m^{k+2}\left((d-1)\frac{1-(-1)^k}2-2m+k+2\right)\left((d-1)\frac{1-(-1)^k}2-2m+k+1\right)\right.\\
& \quad\,+c_m^{k+1}\,(d-4m+2k+1)(d-m+(-1)^k(m-2\lfloor k/2\rfloor-2))\bigg\}\,P_m^{k+2}(d)\\
& =\begin{cases}
\!\begin{alignedat}{2}
& \left\{c_m^{k+2}\,(2m-k-2)(2m-k-1)+c_m^{k+1}\,(d-4m+2k+1)(d-k-2)\right\}\\
& \times P_m^{k+2}(d) & \quad & (k\mbox{ even}),\\
& \left\{(d-2m+k+1)\left(c_m^{k+2}\,(d-2m+k)+c_m^{k+1}\,(d-4m+2k+1)\right)\right\}\\
& \times P_m^{k+2}(d) & \quad & (k\mbox{ odd}).
\end{alignedat}
\end{cases}
\end{align*}
On the other hand, it immediately follows from \eqref{eq_exp-Pmk} that
\[P_{m+1}^{k+2}(d)=(d-2(m-\lfloor(k+1)/2\rfloor))\,P_m^{k+2}(d)=\begin{cases}
(d-2m+k)\,P_m^{k+2}(d) & (k\mbox{ even}),\\
(d-2m+k+1)\,P_m^{k+2}(d) & (k\mbox{ odd}).
\end{cases}\]
Therefore, by Lemma~\ref{lem_tech}(2) and \eqref{eq_exp-cmk}, we obtain for even $k$ that
\begin{align*}
& \quad\,\,c_m^{k+2}\,(2m-k-2)(2m-k-1)+c_m^{k+1}\,(d-4m+2k+1)(d-k-2)\\
& =c_m^{k+1}\,\{2(m-k-1)(2m-k-1)+(d-4m+2k+1)(d-k-2)\}\\
& =c_m^{k+1}\,(d-2m-1)(d-2m+k)=c_{m+1}^{k+2}\,(d-2m-1)(d-2m+k),
\end{align*}
and parallelly for odd $k$ that
\begin{align*}
& \quad\,\,(d-2m+k+1)\left(c_m^{k+2}\,(d-2m+k)+c_m^{k+1}\,(d-4m+2k+1)\right)\\
& =(d-2m+k+1)\left\{\left((k+1)c_m^{k+2}-2(m-k-1)c_m^{k+1}\right)+(c_m^{k+1}+c_m^{k+2})(d-2m-1)\right\}\\
& =c_{m+1}^{k+2}\,(d-2m-1)(d-2m+k+1).
\end{align*}
In other words, we balance the both sides of \eqref{eq_claim-3}, which finishes the proof.\hfill$\square$\smallskip

At this stage, the main conclusion of multiple hyperbolic systems degenerates to a straightforward corollary of the above result.\smallskip

\noindent{\it Proof of Theorem~$\ref{thm_gov-u}$.} In sense of Proposition~\ref{prop_gov-Um}, it suffices to show by induction on $m$ that for $m=0,1,\ldots$ and $d\ge2m+1$, there holds
\begin{equation}\label{eq_induct}
\begin{cases}
\!\begin{alignedat}{2}
& \mathcal P_\rho^{m+1}u=\begin{cases}
\!\begin{alignedat}{2}
& (2m)!!\,2^{m+1}\pi^m\,\alpha/\rho & \quad & (d=2m+1),\\
& \prod_{k=0}^m(d-2k-1)\,U_{m+1}(\,\cdot\,,R(\,\cdot\,)) & & (d>2m+1)
\end{alignedat}
\end{cases} & & \mbox{in }\mathbb R^d\times\mathbb R_+,\\
& \partial_t^ju(\,\cdot\,,0)=0,\quad j=0,1,\ldots,2m+1 & & \mbox{in }\mathbb R^d.
\end{alignedat}
\end{cases}
\end{equation}
The result for $m=0$ was obtained in Lemma~\ref{lem_single-hyper}(2). In order to verify \eqref{eq_induct} for each $m\ge1$, we suppose the validity for some $m-1$, especially there holds for $d\ge2m+1$ that
\[\begin{cases}
\!\begin{alignedat}{2}
& \mathcal P_\rho^mu(\bm x,t)=\prod_{k=0}^{m-1}(d-2k-1)\,U_m(\bm x,R(t)) & \quad & (\bm x\in\mathbb R^d,\ t>0),\\
& \partial_t^ju(\bm x,0)=0,\quad j=0,1,\ldots,2m-1 & & (\bm x\in\mathbb R^d).
\end{alignedat}
\end{cases}\]
Taking $w:=\mathcal P_\rho^mu$, then in view of \eqref{eq_change-var-id} we find
\[W(\bm x,\tau)=w(\bm x,R^{-1}(\tau))=\left.\mathcal P_\rho^mu(\bm x,t)\right|_{t=R^{-1}(\tau)}=\prod_{k=0}^{m-1}(d-2k-1)\,U_m(\bm x,\tau),\]
where $U_m$ satisfies \eqref{eq_gov-Um} by Proposition~\ref{prop_gov-Um}. This, together with \eqref{eq_change-var-id}, yields immediately
\begin{align*}
\mathcal P_\rho^{m+1}u(\bm x,t) & =\mathcal P_\rho w(\bm x,t)=\square W(\bm x,R(t))=\prod_{k=0}^{m-1}(d-2k-1)\,\square U_m(\bm x,R(t))\\
& =\begin{cases}
\!\begin{alignedat}{2}
& \prod_{k=0}^{m-1}(d-2k-1)\,2^{m+1}\pi^m\,F(\bm x,R(t)) & \quad & (d=2m+1)\\
& \prod_{k=0}^{m-1}(d-2k-1)\,(d-2m-1)\,U_{m+1}(\bm x,R(t)) & & (d>2m+1)
\end{alignedat}
\end{cases}\\
& =\begin{cases}
\!\begin{alignedat}{2}
& (2m)!!\,2^{m+1}\pi^m\,\alpha(\bm x,t)/\rho(t) & \quad & (d=2m+1)\\
& \prod_{k=0}^m(d-2k-1)\,U_{m+1}(\bm x,R(t)) & & (d>2m+1)
\end{alignedat}
\end{cases}(\bm x\in\mathbb R^d,\ t>0),
\end{align*}
while the initial condition for $d\ge2m+1$ reads
\begin{align*}
\mathcal P_\rho^mu(\bm x,0) & =\prod_{k=0}^{m-1}(d-2k-1)\,U_m(\bm x,0)=0,\\
\partial_t\mathcal P_\rho^mu(\bm x,0) & =\rho(0)\prod_{k=0}^{m-1}(d-2k-1)\,\partial_\tau U_m(\bm x,0)=0.
\end{align*}
Since $\rho(0)\ne0$ and $\mathcal P_\rho^mu(\bm x,0)$ is now a linear combination of $\partial_t^ju(\bm x,0)$ ($j=0,\ldots,2m-1,2m$), it follows from the inductive assumption on the homogeneous initial condition for lower order time derivatives than $2m$ that $\partial_t^{2m}u(\bm x,0)=0$ and thus $\partial_t^{2m+1}u(\bm x,0)=0$ ($\bm x\in\mathbb R^d$). This completes the demonstration of \eqref{eq_induct} for $m\ge1$ and hence Theorem~\ref{thm_gov-u}.\hfill$\square$

\Section{Numerical Simulations for Forward Problems}\label{sec_numer}

In this section, we implement numerical computations for forward problems in practical dimensions, namely, solving for the expectation number $u(\bm x,t)$ of transformation events by given (discrete) data of $\alpha(\bm x,t)$ and $\rho(t)$ with $d=1,2,3$. It will be demonstrated that even finite difference schemes for the derived hyperbolic-type systems can dramatically improve the efficiency of simulations compared with direct approaches based on \eqref{eq_def-u}--\eqref{eq_def-timecone}.

Throughout this section, we consider the systems in the time interval $[0,T]$ and assume the periodicity of $u(\bm x,t)$ in space. More precisely, it is supposed, e.g. for $d=3$, that there exist $L_\ell>0$ ($\ell=1,2,3$) such that
\[u(x_1+i\,L_1,x_2+j\,L_2,x_3+k\,L_3,t)=u(x_1,x_2,x_3,t)\quad(\forall\,i,j,k\in\mathbb Z),\]
so that it suffices to restrict the systems in $\prod_{\ell=1}^d[0,L_\ell]\times[0,T]$ ($d=1,2,3$) and impose the periodic boundary conditions. Thus the data of $\alpha$ and $\rho$ are assigned only on the knots
\[0=t_0<t_1<\cdots<t_{N_t}=T,\quad0=x_\ell^1<x_\ell^2<\cdots<x_\ell^{N_\ell}=L_\ell\ (1\le\ell\le d).\]
Without lose of generality we assume an equidistant lattice in space, that is, $x_\ell^i=(i-1)\,\Delta s\ (i=1,\ldots,N_\ell)$ with the step size $\Delta s>0$.

Although one may solve for the unknown $u$ by, e.g., hyperbolic-type systems \eqref{eq_gov-u-odd} with the periodic boundary condition when $d=1,3$, it is advantageous to consider the wave-type systems \eqref{eq_gov-Um} for $U_m$ and utilize the relation $u(\,\cdot\,,t)=U_0(\,\cdot\,,R(t))$ instead. In this manner we can not only circumvent the numerical differentiation problem for $\rho$, but also simplify the choice of the step length in time.

Such a consideration of the equivalent systems relies obviously on the knowledge of the change of variable $\tau=R(t)=\int_0^t\rho(s)\,\mathrm ds$, whose accurate value is absent due to the discrete data $\{\rho(t_n)\}_{n=0}^{N_t}$. Hence we shall first apply, for instance, a composite trapezoid quadrature to provide a piecewise linear approximation of $R(t)$, say $\widehat R(t)$. Now that $U_0$ satisfies a wave equation with the unit propagation speed, we may partition the alternative time interval $[0,\widehat R(T)]$ of $\tau$ by a uniform step length $\Delta\tau>0$, yielding the knots $\tau_n=n\,\Delta\tau\ (n=0,1,\ldots,N_\tau)$ with $N_\tau\,\Delta\tau=\widehat R(T)$. Note that $\Delta\tau$ is required to satisfy the Courant-Friedrichs-Lewy condition $\sqrt d\,\Delta\tau\le\Delta s$ when using an explicit scheme of the finite difference method, which can be loosen or removed if some weighted multilevel schemes are employed. For later use we introduce the ratio $r:=(\Delta\tau/\Delta s)^2$.

Thanks to the strict positivity of $\rho$, it is easy to find out an increasing sequence $\{\widehat t_n\}_{n=0}^{N_\tau}$ such that $\widehat R(\widehat t_n)=\tau_n$. Thus, the estimation of $U_0(\,\cdot\,,\tau_n)$ stands for a reasonable approximation of $u(\,\cdot\,,\widehat t_n)$ due to the relation \eqref{eq_def-U0}. Moreover, we recognize that the equidistant partition in $\tau$ corresponds with a self-adaptive partition in $t$, i.e., the knots $\{\widehat t_n\}$ accumulate where $\rho$ is large while are sparsely distributed for small $\rho$ (see Figure~\ref{fig_self-adaptive}).
\begin{figure}[htbp]\centering
\includegraphics[width=.9\textwidth]{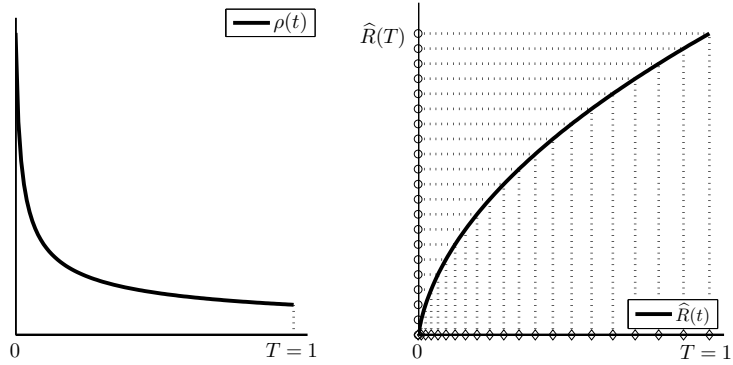}\\
\caption{\footnotesize{An example of the self-adaptiveness. Left figure: plot of $\rho(t)=(t+0.01)^{-1/2}$. Right figure: plot of $\widehat R(t)\approx2(\sqrt{t+0.01}-0.1)$, $\mbox{circle}=\tau_n$ and $\mbox{diamond}=\widehat t_n$.}}\label{fig_self-adaptive}
\end{figure}

Interpreting $\alpha$ and $\rho$ as piecewise linear, we may obtain the term \[F(\,\cdot\,,\tau_n)=\frac{\alpha(\,\cdot\,,R^{-1}(\tau_n))}{\rho(R^{-1}(\tau_n))}\approx\frac{\alpha(\,\cdot\,,\widehat t_n)}{\rho(\widehat t_n)}\quad(0\le n\le N_\tau)\]
by interpolating the discrete data $\{\alpha(\,\cdot\,,t_n),\rho(t_n)\}_{n=0}^{N_t}$. For $d=3$, we denote $F_n^{i,j,k}:=\alpha(x_1^i,x_2^j,x_3^k,\widehat t_n)/\rho(\widehat t_n)$, let $U_{0,n}^{i,j,k}$ be the approximation of $U_0(x_1^i,x_2^j,x_3^k,\tau_n)$, and define the difference operators
\[\begin{cases}
\delta_{x_1}^2U_{0,n}^{i,j,k}:= U_{0,n}^{i+1,j,k}-2\,U_{0,n}^{i,j,k}+U_{0,n}^{i-1,j,k} & (i=1,\ldots,N_1),\\
\delta_{x_2}^2U_{0,n}^{i,j,k}:= U_{0,n}^{i,j+1,k}-2\,U_{0,n}^{i,j,k}+U_{0,n}^{i,j-1,k} & (j=1,\ldots,N_2),\\
\delta_{x_3}^2U_{0,n}^{i,j,k}:= U_{0,n}^{i,j,k+1}-2\,U_{0,n}^{i,j,k}+U_{0,n}^{i,j,k-1} & (k=1,\ldots,N_3),
\end{cases}\]
where we understand $U_{0,n}^{0,j,k}=U_{0,n}^{N_1,j,k}$, $U_{0,n}^{N_1+1,j,k}=U_{0,n}^{1,j,k}$, etc. due to the periodicity. Similar notations are parallelly shared by the counterparts of $U_0$ and $\alpha$ for $d=1,2$ as well as $U_1$ for $d=2,3$.

Now we are well-prepared to explain the implementation of numerical approaches starting from $d=1$. Applying a three-leveled finite difference scheme to \eqref{eq_gov-U0} with the periodic boundary condition, we obtain
\begin{equation}\label{eq_numer-1D}
\begin{cases}
\!\begin{aligned}
& U_{0,n+1}^i-2\,U_{0,n}^i+U_{0,n-1}^i=r\,\delta_{x_1}^2\left(\eta\,U_{0,n+1}^i-(1-2\eta)\,U_{0,n}^i+\eta\,U_{0,n-1}^i\right)\\
& \qquad\qquad\qquad\qquad\qquad\quad\;\;\;+2\,\Delta\tau^2\,F_n^i\quad(1\le i\le N_1,1\le n\le N_\tau-1),\\
& U_{0,0}^i=U_{0,1}^i=0\quad(1\le i\le N_1),\\
& U_{0,n}^0=U_{0,n}^{N_1},U_{0,n}^{N_1+1}=U_{0,n}^1\quad(1\le n\le N_\tau-1),
\end{aligned}
\end{cases}
\end{equation}
where $\eta\in[0,1/2]$ is a parameter. \eqref{eq_numer-1D} becomes the von Neumann scheme when $\eta=1/4$, and it is unconditionally stable as long as $\eta\ge1/4$. The numerical result with $\eta=1/4$, $T=1$, $L_1=\pi$ and the given
\begin{equation}\label{eq_data-1D}
\rho(t)=\frac1{2\sqrt{t+1}},\quad\alpha(x,t)=\exp\left(-\frac{(x-\pi/2)^2}2\right)(1-\cos(10\,x))\exp\left(1-\frac t{10}\right)
\end{equation}
is illustrated in Figure~\ref{fig_numer-1D}.
\begin{figure}[htbp]\centering
\includegraphics[width=.6\textwidth]{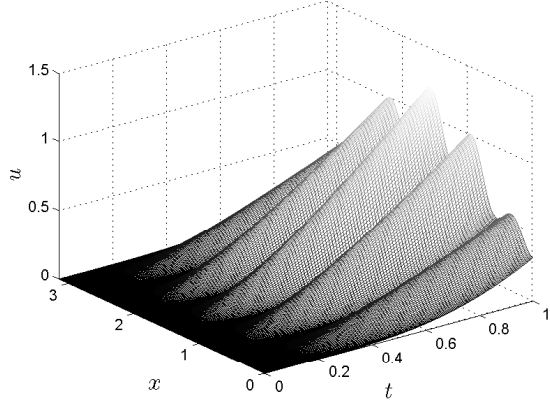}\\
\caption{\footnotesize{Numerical result of the one-dimensional forward problem with $\alpha$ and $\rho$ given by \eqref{eq_data-1D}.}}\label{fig_numer-1D}
\end{figure}

Now we consider the two-dimensional case. In Remark~\ref{rem_Duhamel} we mentioned the different situations between even and odd spatial dimensions, and such difference results in practical difficulties in the treatment for $d=2$. After a polar coordinate transform, the source term in the governing equation \eqref{eq_gov-U0} reads
\begin{align*}
(d-1)\,U_1(\bm x,\tau) & =[1,S_2,\triangle^0](x_1,x_2,\tau)\\
& =\int_0^\tau\!\!\int_0^{2\pi}F(x_1+(\tau-\zeta)\cos\theta,x_2+(\tau-\zeta)\sin\theta,\zeta)\,\mathrm d\theta\mathrm d\zeta,
\end{align*}
that is, the integral on the lateral of the cone $\Omega_1(\bm x,\tau)$. At the moment it is necessary to discretize the above integral at all grid points by, for example, composite trapezoid quadratures and linear interpolation techniques to provide reasonable approximations. Here we omit the details and just suppose that each $U_{1,n}^{i,j}\approx U_1(x_1^i,x_2^j,\tau_n)$ is obtained. To avoid massive matrix manipulations while preserve the unconditional stability as the von Neumann scheme in one-dimensional case, we deal with system \eqref{eq_gov-U0} by the alternating direction implicit (ADI) method (see Lees~\cite{L-1962})
\[\begin{cases}
\!\begin{aligned}
& U_{0,n+1/2}^{i,j}-2\,U_{0,n}^{i,j}+U_{0,n-1}^{i,j}=r\,\delta_x^2\left(\eta\,U_{0,n+1/2}^{i,j}+(1-2\eta)\,U_{0,n}^{i,j}+\eta\,U_{0,n-1}^{i,j}\right)\\
& \qquad\qquad\qquad\qquad\qquad\qquad\:\:\:+r\,\de_y^2U_{0,n}^{i,j}+\Delta\tau^2\,U_{1,n}^{i,j},\\
& U_{0,n+1}^{i,j}-U_{0,n+1/2}^{i,j}=r\eta\,\delta_y^2\left(U_{0,n+1}^{i,j}-2\,U_{0,n}^{i,j}+U_{0,n-1}^{i,j}\right)
\end{aligned}
\end{cases}\]
with $\eta\in[0,1/2]\ (1\le i\le N_1,1\le j\le N_2,1\le n\le N_\tau-1)$, which inherits the unconditional stability property when $\eta\ge1/4$. Here the initial and boundary treatments are parallel to that of \eqref{eq_numer-1D}, and the notation $U_{0,n+1/2}^{i,j}$ only stands for an intermediate procedure in pursue of $U_{0,n+1}^{i,j}$ instead of any estimation at $\tau=(n+1/2)\,\Delta\tau$. We implement the above scheme with $\eta=1/4$, $T=50$, $L_1=L_2=1$ and the given
\begin{equation}\label{eq_data-2D}
\rho(t)=\frac1{50\sqrt{t+1}},\quad\alpha(\bm x,t)=f(\bm x)\left(1-\exp\left(-\frac t{10}\right)\right).
\end{equation}
Here the spatial component $f$ of $\alpha$ describes a hexagon-shaped structure satisfying the periodicity with the addition of a random noise subjected to the Cauchy distribution, producing few outstanding pixels with a low-amplitude background (see Kaipio \& Somersalo~\cite[\S3.3.2]{KS-2005}). In Figure~\ref{fig_numer-2D}, we capture several cuts at different stages of the phase transformation.
\begin{figure}[htbp]\centering
\includegraphics[width=\textwidth]{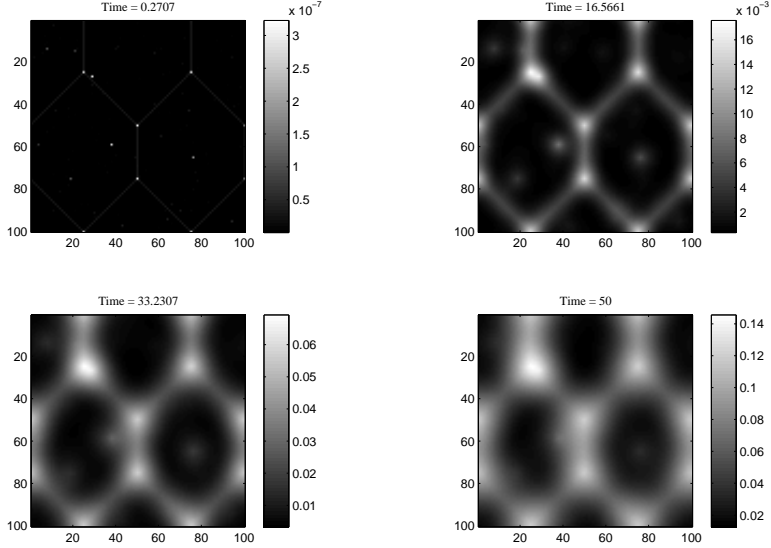}\\
\caption{\footnotesize{Numerical simulation of the two-dimensional forward problem with $\alpha$ and $\rho$ given by \eqref{eq_data-2D}.}}\label{fig_numer-2D}
\end{figure}

Finally, for $d=3$ it suffices to treat \eqref{eq_gov-Um} with $m=1$ and $m=0$ sequentially so that we first solve for $U_1$ by the data $4\pi\,F$ and then obtain $U_0$ by the source term $2\,U_1$. As for the numerical scheme, we apply the three-dimensional version of the ADI method (see Fairweather \& Metchell~\cite{FM-1965}) to both $U_1$ and $U_0$ as
\begin{align*}
& \begin{cases}
\!\begin{aligned}
& U_{1,n+1/3}^{i,j,k}-2\,U_{1,n}^{i,j,k}+U_{1,n-1}^{i,j,k}=r\,\delta_x^2\left(\eta\,U_{1,n+1/3}^{i,j,k}+(1-2\eta)\,U_{1,n}^{i,j,k}+\eta\,U_{1,n-1}^{i,j,k}\right)\\
& \qquad\qquad\qquad\qquad\qquad\qquad\quad\,+r\,(\delta_y^2+\delta_z^2)U_{1,n}^{i,j,k}+4\pi\,\Delta\tau^2\,F_n^{i,j,k},\\
& U_{1,n+2/3}^{i,j,k}-U_{1,n+1/3}^{i,j,k}=r\eta\,\delta_y^2\left(U_{1,n+2/3}^{i,j,k}-2\,U_{1,n}^{i,j,k}+U_{1,n-1}^{i,j,k}\right),\\
& U_{1,n+1}^{i,j,k}-U_{1,n+2/3}^{i,j,k}=r\eta\,\delta_z^2\left(U_{1,n+1}^{i,j,k}-2\,U_{1,n}^{i,j,k}+U_{1,n-1}^{i,j,k}\right),
\end{aligned}
\end{cases}\\
& \begin{cases}
\!\begin{aligned}
& U_{0,n+1/3}^{i,j,k}-2\,U_{0,n}^{i,j,k}+U_{0,n-1}^{i,j,k}=r\,\delta_x^2\left(\eta\,U_{0,n+1/3}^{i,j,k}+(1-2\eta)\,U_{0,n}^{i,j,k}+\eta\,U_{0,n-1}^{i,j,k}\right)\\
& \qquad\qquad\qquad\qquad\qquad\qquad\quad\,+r\,(\delta_y^2+\delta_z^2)U_{0,n}^{i,j,k}+2\,\Delta\tau^2\,U_{1,n}^{i,j,k},\\
& U_{0,n+2/3}^{i,j,k}-U_{0,n+1/3}^{i,j,k}=r\eta\,\delta_y^2\left(U_{0,n+2/3}^{i,j,k}-2\,U_{0,n}^{i,j,k}+U_{0,n-1}^{i,j,k}\right),\\
& U_{0,n+1}^{i,j,k}-U_{0,n+2/3}^{i,j,k}=r\eta\,\delta_z^2\left(U_{0,n+1}^{i,j,k}-2\,U_{0,n}^{i,j,k}+U_{0,n-1}^{i,j,k}\right)
\end{aligned}
\end{cases}
\end{align*}
with $\eta\in[0,1/2]\ (1\le i\le N_1,1\le j\le N_2,1\le k\le N_3,1\le n\le N_\tau-1)$. Still the stability properties of such an approach are identical to that of lower dimensions.

\Section{Conclusion and Future Works}\label{sec_cclu}

In summary, it reveals that Cahn's time cone model \eqref{eq_def-u}--\eqref{eq_def-timecone} concerning phase transformation kinetics can be equivalently described by a class of multiple hyperbolic systems with the homogeneous initial condition, in which the growth speed $\rho(t)$ mainly plays the role of the propagation speed of wave. Especially, such systems take the simplest forms in odd spatial dimensions, where the nucleation rate $\alpha(\bm x,t)$ accounts for the source term (see Theorem~\ref{thm_gov-u}). Moreover, by the change of variable \eqref{eq_change-var} which only involves $\rho(t)$, the governing equation \eqref{eq_gov-u-odd} is further reduced to a multiple d'Alembertian system with unit propagation speed (see Proposition~\ref{prop_gov-Um}). To a certain extent, the derivation of hyperbolic-type governing equations provides an appropriate formulation which enables systematic investigations of problems related to structure transformations in both theoretical and numerical senses. As a tentative application, it was demonstrated in the previous section that efficient forward solvers are readily implemented on basis of this alternative framework instead of the original model.

More significantly, the transform from an integral equation to partial differential equations also initiates smooth discussions on the corresponding inverse problems by using classical results of inverse hyperbolic problems. In a forthcoming paper, we shall study the problem of identifying the nucleation rate $\alpha(\bm x,t)$ by several kinds of observation data of the generation events $u(\bm x,t)$, for instance, by final measurements and partial interior measurements. It turns out that the reasoning can be easily carried out from a viewpoint of inverse source problems of the hyperbolic type. On the other hand, the reconstruction of the growth speed $\rho(t)$ may be possible by regarding it as either the wave speed or the function determining a change of variable. More challenging topics may involve the simultaneous identification of both $\alpha$ and $\rho$ by more informative observations, as well as the computational methods for the above mentioned inverse problems.

\appendix
\Section{Technical Details}\label{sec_tech}

Here we provide detailed proofs of the technical lemmata in Section~\ref{sec_proof}.\smallskip

\noindent{\it Proof of Lemma~$\ref{lem_brkt-diff}$.} For the boundary integral $[k,S_d,\triangle^j]$, we introduce the polar transform
\begin{align*}
& \bm y=\bm x+(\tau-\zeta)\,\bm p(\bm\varphi)\quad(\bm\varphi=(\varphi_1,\ldots,\varphi_{d-2},\varphi_{d-1})\in D_d:=[0,\pi]^{d-2}\times[0,2\pi]),\nonumber\\
& \bm p(\bm\varphi):=(\cos\varphi_1,\sin\varphi_1\cos\varphi_2,\ldots,\sin\varphi_1\cdots\sin\varphi_{d-2}\cos\varphi_{d-1},\sin\varphi_1\cdots\sin\varphi_{d-2}\sin\varphi_{d-1}).\label{eq_polar-angle}
\end{align*}
Then the Jacobian reads $(\tau-\zeta)^{d-1}\,q(\bm\varphi)$, where $q(\bm\varphi):=\sin^{d-2}\varphi_1\cdots\sin\varphi_{d-2}$. Therefore, we can write expression \eqref{eq_def-brkt-bdry} equivalently as
\[[k,S_d,\triangle^j](\bm x,\tau)=\int_0^\tau\!\!\int_{D_d}(\tau-\zeta)^{d-k-1}\,q(\bm\varphi)\,\triangle^jF(\bm x+(\tau-\zeta)\,\bm p(\bm\varphi),\zeta)\,\mathrm d\bm\varphi\mathrm d\zeta.\]
Parallelly, for the interior integral $[k,B_d,\triangle^j]$, we apply a similar polar transform
\[\bm y=\bm x+\ell\,\bm p(\bm\varphi)\quad(0<\ell<\tau-\zeta,\ \bm\varphi=(\varphi_1,\ldots,\varphi_{d-2},\varphi_{d-1})\in D_d)\]
with the Jacobian $\ell^{d-1}\,q(\bm\varphi)$, where $\bm p(\bm\varphi)$, $D_d$ and $q(\bm\varphi)$ are defined as before. Then \eqref{eq_def-brkt-inter} can be rewritten as
\[[k,B_d,\triangle^j](\bm x,\tau)=\int_0^\tau\!\!\int_0^{\tau-\zeta}\!\!\!\int_{D_d}\frac{\ell^{d-1}\,q(\bm\varphi)\,\triangle^jF(\bm x+\ell\,\bm p(\bm\varphi),\zeta)}{(\tau-\zeta)^k}\,\mathrm d\bm\varphi\mathrm d\ell\mathrm d\zeta.\]

With these alternative representations, it is straightforward to verify \eqref{eq_brkt-laplace} that
\begin{align*}
\triangle[k,S_d,\triangle^j](\bm x,\tau) & =\int_0^\tau\!\!\int_{D_d}(\tau-\zeta)^{d-k-1}\,q(\bm\varphi)\,\triangle^{j+1}F(\bm x+(\tau-\zeta)\,\bm p(\bm\varphi),\zeta)\,\mathrm d\bm\varphi\mathrm d\zeta\\
& =[k,S_d,\triangle^{j+1}](\bm x,\tau),\\
\triangle[k,B_d,\triangle^j](\bm x,\tau) & =\int_0^\tau\!\!\int_0^{\tau-\zeta}\!\!\!\int_{D_d}\frac{\ell^{d-1}\,q(\bm\varphi)\,\triangle^{j+1}F(\bm x+\ell\,\bm p(\bm\varphi),\zeta)}{(\tau-\zeta)^k}\,\mathrm d\bm\varphi\mathrm d\ell\mathrm d\zeta\\
& =[k,B_d,\triangle^{j+1}](\bm x,\tau).
\end{align*}
For $[k,S_d,\triangle^j]$ with $k<d-1$, we apply Green's formula and notice the fact that $\bm p(\bm\varphi)$ coincides with the unit outward normal vector $\bm\nu(\bm y)$ at $\bm y=\bm x+(\tau-\zeta)\,\bm p(\bm\varphi)$ to proceed
\begin{align*}
& \quad\,\,\partial_\tau[k,S_d,\triangle^j](\bm x,\tau)\\
& =\int_0^\tau\partial_\tau\left(\int_{D_d}(\tau-\zeta)^{d-k-1}\,q(\bm\varphi)\,\triangle^jF(\bm x+(\tau-\zeta)\,\bm p(\bm\varphi),\zeta)\,\mathrm d\bm\varphi\right)\mathrm d\zeta\\
& =(d-k-1)\int_0^\tau\!\!\int_{D_d}(\tau-\zeta)^{d-k-2}\,q(\bm\varphi)\,\triangle^jF(\bm x+(\tau-\zeta)\,\bm p(\bm\varphi),\zeta)\,\mathrm d\bm\varphi\mathrm d\zeta\\
& \quad\,+\int_0^\tau\!\!\int_{D_d}(\tau-\zeta)^{d-k-1}\,q(\bm\varphi)\,\nabla\triangle^jF(\bm x+(\tau-\zeta)\,\bm p(\bm\varphi),\zeta)\cdot\bm p(\bm\varphi)\,\mathrm d\bm\varphi\mathrm d\zeta\\
& =(d-k-1)\,[k+1,S_d,\triangle^j](\bm x,\tau)+\int_0^\tau\!\!\int_{S_d(\bm x,\tau-\zeta)}\frac{\nabla\triangle^jF(\bm y,\zeta)\cdot\bm\nu(\bm y)}{(\tau-\zeta)^k}\,\mathrm d\bm\sigma\mathrm d\zeta\\
& =(d-k-1)\,[k+1,S_d,\triangle^j](\bm x,\tau)+\int_0^\tau\!\!\int_{B_d(\bm x,\tau-\zeta)}\frac{\triangle^{j+1}F(\bm y,\zeta)}{(\tau-\zeta)^k}\,\mathrm d\bm y\mathrm d\zeta\\
& =(d-k-1)\,[k+1,S_d,\triangle^j](\bm x,\tau)+[k,B_d,\triangle^{j+1}](\bm x,\tau).
\end{align*}
For $k=d-1$, a similar argument yields immediately
\begin{align*}
\partial_\tau[d-1,S_d,\triangle^j](\bm x,\tau) & =\int_{D_d}q(\bm\varphi)\,\triangle^jF(\bm x,\tau)\,\mathrm d\bm\varphi\\
& \quad\,+\int_0^\tau\!\!\int_{D_d}q(\bm\varphi)\,\nabla\triangle^jF(\bm x+(\tau-\zeta)\,\bm p(\bm\varphi),\zeta)\cdot\bm p(\bm\varphi)\,\mathrm d\bm\varphi\mathrm d\zeta\\
& =\sigma_d\,\triangle^jF(\bm x,\tau)+\int_0^\tau\!\!\int_{S_d(\bm x,\tau-\zeta)}\frac{\nabla\triangle^jF(\bm y,\zeta)\cdot\bm\nu(\bm y)}{(\tau-\zeta)^{d-1}}\,\mathrm d\bm\sigma\mathrm d\zeta\\
& =\sigma_d\,\triangle^jF(\bm x,\tau)+\int_0^\tau\!\!\int_{B_d(\bm x,\tau-\zeta)}\frac{\triangle^{j+1}F(\bm y,\zeta)}{(\tau-\zeta)^{d-1}}\,\mathrm d\bm y\mathrm d\zeta\\
& =\sigma_d\,\triangle^jF(\bm x,\tau)+[d-1,B_d,\triangle^{j+1}](\bm x,\tau),
\end{align*}
which is indeed \eqref{eq_brkt-timediff-bdry}. For $[k,B_d,\triangle^j]$ with $k\le d-1$, we employ a parallel calculation to derive \eqref{eq_brkt-timediff-inter} as
\begin{align*}
\partial_\tau[k,B_d,\triangle^j](\bm x,\tau) & =\int_0^\tau\partial_\tau\left(\int_0^{\tau-\zeta}\!\!\!\int_{D_d}\frac{\ell^{d-1}\!q(\bm\varphi)\,\triangle^jF(\bm x+\ell\,\bm p(\bm\varphi),\zeta)}{(\tau-\zeta)^k}\,\mathrm d\bm\varphi\mathrm d\ell\right)\mathrm d\zeta\\
& =-k\int_0^\tau\!\!\int_0^{\tau-\zeta}\!\!\!\int_{D_d}\frac{\triangle^jF(\bm x+\ell\,\bm p(\bm\varphi),\zeta)}{(\tau-\zeta)^{k+1}}\,\mathrm d\bm\varphi\mathrm d\ell\mathrm d\zeta\\
& \quad\,+\int_0^\tau\!\!\int_{D_d}(\tau-\zeta)^{d-k-1}\,q(\bm\varphi)\,\triangle^jF(\bm x+(\tau-\zeta)\,\bm p(\bm\varphi),\zeta)\,\mathrm d\bm\varphi\mathrm d\zeta\\
& =-k\,[k+1,B_d,\triangle^j](\bm x,\tau)+[k,S_d,\triangle^j](\bm x,\tau).
\end{align*}
The proof is completed.\hfill$\square$\smallskip

\noindent{\it Proof of Lemma~$\ref{lem_tech}$.} We proceed for both assertions by induction on $m$.

(1)\ \ For $m=2$, \eqref{eq_exp-Pmk} reads $P_2^1(d)=d-2$ and $P_2^2(d)=1$, indicating \eqref{eq_rel-Pmk} immediately by taking $m=k=2$. Supposing \eqref{eq_rel-Pmk} holds for some $m\ge2$, we shall show that it still holds for $m+1$, namely
\[P_{m+1}^{k-1}(d)=((d-m-1)+(-1)^k(m+1-2\lfloor k/2\rfloor))\,P_{m+1}^k(d)\quad(2\le k\le m+1,\ m\ge2).\]
The case $k=m+1$ is trivial, otherwise we replace $m$ by $m+1$ in \eqref{eq_exp-Pmk} and apply the inductive assumption \eqref{eq_rel-Pmk} for $m$ to find
\begin{align*}
P_{m+1}^{k-1}(d) & =(d-2(m+1-\lfloor k/2\rfloor))\,P_m^{k-1}(d)\\
& =(d-2(m+1-\lfloor k/2\rfloor))\,(d-m+(-1)^k(m-2\lfloor k/2\rfloor))\,P_m^k(d),\\
P_{m+1}^k(d) & =(d-2(m+1-\lfloor(k+1)/2\rfloor))\,P_m^k(d).
\end{align*}
As a result, it suffices to show
\begin{align*}
& \quad\,\,(d-2(m+1-\lfloor k/2\rfloor))\,(d-m+(-1)^k(m-2\lfloor k/2\rfloor))\\
& =(d-m-1+(-1)^k(m+1-2\lfloor k/2\rfloor))\,(d-2(m+1-\lfloor(k+1)/2\rfloor)),
\end{align*}
which can be easily verified by discussing the parity of $k$.

(2)\ \ For $m=3$, \eqref{eq_exp-cmk} implies $c_3^2=1$ and \eqref{eq_rel-cmk} follows immediately by taking $m=3$ and $k=2$. Supposing \eqref{eq_rel-cmk} is valid for some $m\ge3$, we shall show that it still holds for $m+1$, namely
\[2(m-k+1)\,c_{m+1}^k=\begin{cases}
k\,c_{m+1}^{k+1} & (k\mbox{ even}),\\
(2m-k+1)\,c_{m+1}^{k+1} & (k\mbox{ odd})
\end{cases}\quad(2\le k\le m,\ m\ge4).\]
For odd $k$, \eqref{eq_exp-cmk} yields $c_{m+1}^k=c_m^{k-1}+c_m^k$ and $c_{m+1}^{k+1}=c_m^k$, while the inductive assumption \eqref{eq_rel-cmk} implies $2(m-k+1)\,c_m^{k-1}=(k-1)\,c_m^k$ since $k-1$ is even. Therefore
\[2(m-k+1)\,c_{m+1}^k=2(m-k+1)\,(c_m^{k-1}+c_m^k)=(2m-k+1)\,c_m^k=(2m-k+1)\,c_{m+1}^{k+1}.\]
Parallelly, we obtain for even $k$ that
\[2(m-k+1)\,c_{m+1}^k=2(m-k+1)\,c_m^{k-1}=(2m-k)\,c_m^k=k\,(c_m^k+c_m^{k+1})=k\,c_{m+1}^{k+1}\]
since now $k-1$ is odd. This ends the proof.\hfill$\square$

\bigskip

{\bf Acknowledgement}\ \ The authors appreciate the invaluable discussions with Professors Jin Cheng, Wenbin Chen, Shuai Lu and Mr. Lingdi Wang during their visits to School of Mathematical Sciences, Fudan University. The authors are also grateful to Professor Vincenzo Capasso for the useful comments.

\end{document}